\documentclass[12pt]{amsart}
\usepackage{amssymb,amscd,amsmath,amsthm}

\newtheorem{thm}{Theorem}
\newtheorem{cor}[thm]{Corollary}

\newtheorem{lemma}[thm]{Lemma}

\newtheorem{example}[thm]{Example}

\begin{document}

\title[RBAs with two nonreal basis elements]{Schur indices for reality-based algebras with two nonreal basis elements}

\author[A. Herman]{Allen Herman$^*$}
\address{Department of Mathematics and Statistics, University of Regina, Regina, Canada, S4S 0A2}
\email{Allen.Herman@uregina.ca}

\thanks{$^*$The work of the first author was been supported by an NSERC Discovery Grant. }
\date{}



\begin{abstract}
This article discusses the representation theory of noncommutative algebras reality-based algebras with positive degree map over their field of definition.  When the standard basis contains exactly two nonreal elements, the main result expresses the noncommutative simple component as a generalized quaternion algebra over its field of definition.   The field of real numbers will always be a splitting field for this algebra, but there are noncommutative table algebras of dimension $6$ with rational field of definition for which it is a division algebra.  The approach has other applications, one of which shows noncommutative association scheme of rank $7$ must have at least three symmetric relations.  
\end{abstract} 

\maketitle

\section{Introduction}

A {\it reality-based algebra} (abbr. RBA) is a pair $(A,\mathbf{B})$, where 

\smallskip
(rba1) $A$ is an associative unital algebra over $\mathbb{C}$ with skew-linear involution $*$; 

\smallskip
(rba2) $\mathbf{B} = \{b_0,b_1,\dots,b_{r-1}\}$ is a distinguished $*$-invariant basis of $A$; 

\smallskip
(rba3) $1_A \in \mathbf{B}$ (we fix $b_0 = 1_A$); 

\smallskip
(rba4) all of the structure constants $\lambda_{ijk}$ given by
$$b_i b_j = \sum_{k=0}^{r-1} \lambda_{ijk} b_k, \mbox{ for } 0 \le i,j,k \le r-1 , \mbox{ are {\it real} numbers; and } $$ 

\smallskip
(rba5) $\mathbf{B}$ satisfies the {\it pseudo-inverse condition}: if we also use $*$ to denote the permutation of order $1$ or $2$ on $\{0,1,\dots,r-1\}$ induced by the action of the involution $*$ on $\mathbf{B}$, then for all $i,j \in \{0,1,\dots,r-1\}$,
$$ \lambda_{ij0} \ne 0 \mbox{ iff } j=i^* \mbox{ and } \lambda_{i^*i0}=\lambda_{ii^*0}>0.$$

\medskip
Axiom (rba4) implies that $\mathbb{R}\mathbf{B}$, the $\mathbb{R}$-span of $\mathbf{B}$, is a $*$-subalgebra of $A$, and axiom (rba5) tells us each $b_i \in \mathbf{B}$ has a unique pseudo-inverse in $\mathbf{B}$.  The dimension $r = |\mathbf{B}|$ is the called the {\it rank} of the RBA.  An element of $\mathbf{B}$ is called {\it real} if it is $*$-invariant, and otherwise it is called {\it nonreal}.  (This terminology reflects the fact that the value of any character of $A$ on a real basis element has to be a real number.  Real basis elements have also been referred to as $*$-symmetric or $*$-invariant in the literature.)  Since $(b_i^*)^*=b_i$ for every $b_i \in \mathbf{B}$, the nonreal basis elements always occur in pairs.    

An RBA {\it has a positive degree map} if there is a complex $*$-algebra representation for which $\delta(b_i) > 0$ for all $b_i \in \mathbf{B}$.   In this case each distinguished basis can be normalized by positive constants so that the coefficient $\lambda_{ii^*0}$ of $b_0=1$ in $b_ib_i^*$ is equal to $\delta(b_i)$, for all $b_i \in \mathbf{B}$.  This basis is unique relative to $\delta$ and we refer to it as the {\it standard} basis of the RBA.  We will assume from now on that the distinguished basis of an RBA with positive degree map is this standard one.  For further background on RBAs, we refer the reader to \cite{Blau09}. 

The {\it field of definition} $F$ of an RBA is the minimal field extension of $\mathbb{Q}$ containing $\Lambda$, where $\Lambda$ is the set of structure constants for the standard basis.   By (rba4) we know this is always a subfield of $\mathbb{R}$.  We will say an RBA is {\it $F$-rational} when its field of definition is $F$, and {\it rational} when its field of definition is $\mathbb{Q}$.  As the structure constants relative to the regular matrices of a finite group or the adjacency matrices of an association scheme are nonnegative integers, finite groups and association schemes give familiar examples of rational RBAs with positive degree map.  

When $F$ is the field of definition for an RBA with positive degree map, $F\mathbf{B}$ will be an $r$-dimensional semisimple $F$-algebra  \cite{Hanaki00}.  A {\it splitting field} of $F\mathbf{B}$ is a field extension $K$ of $F$ for which $K\mathbf{B}$ is a direct sum of full matrix algebras over $K$.  In particular, it is a field $K$ for which every irreducible character of $A$ is afforded by a matrix representation that maps every element of the standard basis $\mathbf{B}$ to a matrix with entries in $K$; i.e.~every  $\chi \in Irr(A)$ is realized over $K$. 

Our main result concerns noncommutative RBAs with positive degree map that have exactly one pair of nonreal standard basis elements.  The argument makes use of the general observation that irreducible characters of RBAs can always be realized up to similarity by $*$-representations; i.e.~those that satisfy $\mathcal{X}(a^*) = \overline{\mathcal{X}(a)^{\top}}$ for all $a \in A$.  We apply Linchenko and Montgomery's extension of Frobenius-Schur indicator theory for algebras with involution \cite{LM2000} to obtain some restrictions on irreducible characters of RBAs with two (and four) nonreal basis elements.  When the RBA $(A,\mathbf{B})$ has two nonreal basis elements, it shows there is a unique irreducible character $\chi$ of degree $2$, all other irreducible characters are of degree $1$, and every irreducible character is realizable over $\mathbb{R}$.  Being the unique irreducible character of degree $2$, $\chi$ takes values in the field of definition $F$.  If $\mathcal{X}$ is a real $*$-algebra representation affording $\chi$, then $\mathcal{X}(F\mathbf{B})$ is a $4$-dimensional simple $F$-algebra, and hence it can be expressed as a generalized quaternion algebra in terms of its symbol $\big( \dfrac{\alpha, \beta}{F} \big)$ for a pair of parameters $\alpha, \beta \in F^{\times}$.  Our main result gives a character-theoretic formula for the symbol of $\mathcal{X}(F\mathbf{B})$.  Although $\mathbb{R}$ is a splitting field, we give examples of some integral RBAs of rank $6$ where this component is a division algebra.   

Along the way we explore some applications of these methods.  For instance, the application of the indicator to noncommutative RBAs with $2$ or $4$ basis elements can be applied to establish a few cases of the ``real Schur index one'' question for irreducible characters that arise in the realization cone of an abstract regular polytope.  In the last section we consider the application of our results to association schemes and noncommutative RBAs of small rank.  Although the only finite groups with exactly two nonreal elements are the dihedral groups of order $6$ and $8$, there are infinite families of table algebras and association schemes with this property.  It applies, for example, to all noncommutative RBAs of rank $5$ and $6$ with positive degree map that were discussed in \cite{HMX17} and \cite{HMX18}.  There are noncommutative RBAs of rank $7$ with either two, four, or six nonreal basis elements.  We will give an example of a noncommutative rational rank $7$ RBA with three pairs of nonreal basis elements for which $\mathcal{X}(\mathbb{R}\mathbf{B})$ is real quaternion algebra.   However, integrality restrictions on the character table show such RBAs can never have algebraic integer structure constants, so noncommutative association schemes with six nonreal basis elements cannot exist.  

\section{Frobenius-Schur indicator theory}   

In this section we show that the field of real numbers is a splitting field for $\mathbb{R}\mathbf{B}$ when $\mathbf{B}$ contains only $2$ nonreal elements.  Intuitively this is somewhat obvious because any basis of the real quaternion algebra $\mathbb{H}$ must have $3$ nonreal elements, but by appealing to Frobenius-Schur indicator theory we can actually give a complete characterization of the real realizability of irreducible characters of $\mathbb{C}\mathbf{B}$ when $\mathbf{B}$ has $2$ or $4$ nonreal elements. 

Frobenius-Schur indicator theory for group algebras over $\mathbb{R}$ was extended to algebras with involution by Linchenko and Montgomery in \cite[Theorem 2.7]{LM2000}.  We will apply their main theorem in the case where $A = \mathbb{C}\mathbf{B}$, where $B$ is a $*$-invariant basis of $A$ whose structure constants are real - these assumptions are equivalent to axioms (rba1), (rba2), and (rba4).  (Here $Tr(S)$ denotes the trace of the linear operator $S$.)  

\begin{thm} \label{LM}
Let $A$ be an $r$-dimensional algebra over $\mathbb{C}$ with skew-linear involution $*$.  Suppose further that $A = \mathbb{C} \mathbf{B}$, where $\mathbf{B}$ is a $*$-invariant basis of $A$ and the structure constants relative to $\mathbf{B}$ are real.  Let $S=\bar{*}$.    Suppose $\{a_i,b_i\}_{i=1}^r$ is a pair of dual bases for $A$ with respect to a symmetric bilinear associative nondegenerate form on $A$.  For all $\chi \in Irr(A)$, define
$$\nu(\chi) = \dfrac{\chi(1)}{\chi(\sum_{i=1}^r a_i b_i)} \sum_{i=1}^r \chi(S(a_i))b_i.$$

Then 

\begin{enumerate} 
\item $\nu(\chi) = 0$, $1$, or $-1$ for all $\chi \in Irr(A)$;  

\item $\nu(\chi) = \begin{cases} 0 & \mbox{ iff $\chi$ is not real-valued; } \\ 1 & \mbox{ if $\chi$ is realized over $\mathbb{R}$; } \\ -1 & \mbox{ if $\chi$ is real-valued but not realizable over $\mathbb{R}$; and } \end{cases}$

\item $Tr (S) = \sum_{\chi \in Irr(A)} \nu(\chi) \chi(1_A)$.
\end{enumerate} 
\end{thm}

Note that since $\mathbf{B}$ has real structure constants, a representation affording $\chi \in Irr(A)$ restricts to a representation of $\mathbb{R}\mathbf{B}$, and in this case the three cases of $\nu(\chi)$ in Theorem \ref{LM} (ii) distinguish when the simple component of $\mathbb{R}\mathbf{B}$ is isomorphic to a full matrix ring over $\mathbb{C}$, $\mathbb{R}$, or $\mathbb{H}$, respectively.  (Note that the statement of (ii) in \cite{LM2000} has to be made in terms of the existence of symmetric or skew-symmetric nondegenerate $A$-invariant bilinear forms because they make no assumptions on the field generated by the structure constants of a basis - see \cite[pg. 348-349]{LM2000}.) 

\begin{lemma} \label{Ess}
Suppose $A$ is a finite-dimensional algebra with involution over $\mathbb{C}$ that has a $*$-invariant basis $\mathbf{B}$ whose structure constants are real. 

\begin{enumerate} 
\item The number of $*$-invariant elements of $\mathbf{B}$ is 
$$s = \sum_{\chi \in Irr(A)} \nu(\chi) \chi(1_A).$$

\item If $A$ has an irreducible character with $\nu(\chi) = -1$, then $s \le r - 6$. 

\item If $s = r-2$, then $A$ has a unique irreducible character $\chi$ with $\chi(1_A) = 2$; all other irreducible characters of $A$ have degree $1$; and every irreducible character of $A$ is realizable over $\mathbb{R}$. 

\item If $s = r-4$, then either (a) $A$ has two irreducible characters of degree $2$, all other irreducible characters of $A$ have degree $1$, and every irreducible character of $A$ is realizable over $\mathbb{R}$; or (b) $A$ has one irreducible character with $\chi(1_A)=2$ that is realizable over $\mathbb{R}$, two complex-valued irreducible characters of degree $1$, and all other irreducible characters are real-valued with degree $1$. 
\end{enumerate}
\end{lemma} 

\begin{proof} 
(i) Since the distinguished basis $\mathbf{B}$ is $*$-closed (and hence $S$-closed), $Tr (S)$ is the number of $*$-invariant basis elements.  Applying Theorem \ref{LM} (iii) proves (i). 

\medskip
(ii) If $r = \dim A$, then $r = \sum_{\psi \in Irr(A)} \psi(1_A)^2$.  If $\chi \in Irr(A)$ with $\nu(\chi)=-1$, then the difference between $r$ and $s$ is at least $\chi(1_A) [\chi(1_A) + 1]$.  Since $\nu(\chi)=-1$ implies the simple component of $\mathbb{R}\mathbf{B}$ corresponding to $\chi$ has even dimension, $\chi(1) \ge 2$.  Therefore, $r - s \ge 6$.  

\medskip
(iii) By (i), we have 
$$ r = \sum_{\psi \in Irr(A)} \psi(1_A)^2 \qquad \mbox{ and } \qquad s = r-2 = \sum_{\psi \in Irr(A)} \nu(\psi) \psi(1_A).$$  The gap between $\sum_{\psi} \psi(1_A)^2$ and $\sum_{\psi} \nu(\psi) \psi(1_A)$ is minimized when $\nu(\psi) = 1$ for every $\psi \in Irr(A)$, and this minimal gap is equal to $\sum_{\psi(1_A)>1} (\psi(1_A)^2 - \psi(1_A))$.  Since this is exactly $2$, we can conclude that there is only one irreducible character of degree $>1$, that this irreducible character has degree $2$, and that $\nu(\psi) = 1$ for every $\psi \in Irr(A)$. The latter implies $\psi$ is realizable over $\mathbb{R}$ for every $\psi \in Irr(A)$.  

\medskip
(iv) Since $A$ is noncommutative there is at least one irreducible character with $\chi(1) \ge 2$.  For every irreducible character with $\nu(\chi)=0$ the number of nonreal elements increases by $\chi(1)^2$.  Since the irreducible characters with $\nu(\chi)=0$ come in complex conjugate pairs, we can have such a pair of irreducibles only when there is just one such pair of degree $1$ and there is one other irreducible of degree $2$ that has $\nu(\chi)=1$.  In this case all other irreducible characters must have $\chi(1) = 1$ and $\nu(\chi)=1$.  This situation is case (b). 

Now suppose every irreducible character has $\nu(\chi)=1$.  Then $$r-s = \sum_{\chi \in Irr(A)} [\chi(1)^2-\chi(1)].$$  This number can be $4$ only when there are two irreducibles $\chi$ with degree $2$ and all others have degree $1$.  This situation is case (a). 
\end{proof}

\begin{example} {\rm
Lemma \ref{Ess} can sometimes be applied to answer a question about Schurian association schemes arising naturally in the study of abstract regular polytopes - see \cite[Problem 23]{Schulte-Weiss}.  This question asks if $\nu(\chi) \ne 1$ for every irreducible character of the adjacency algebra of the Schurian association scheme corresponding to the $H$-$H$-double cosets of a finite string $C$-group $G$ relative to its first vertex stabilizer subgroup $H$.  Since this is known for all finite Coxeter groups, one must consider finite groups that are homomorphic images of infinite Coxeter groups.  The question is answered affirmatively by Lemma \ref{Ess} for the particular group $G$ whenever the set $G/\!\!/H$ of $H$-$H$-double cosets in $G$ contains two or four double cosets with $g^{-1} \not\in HgH$. We have found some such Schurian association schemes among rank $3$ string $C$-groups, which are characterized by their Schlafi type $[m_1,m_2]_{k,\ell}$ for positive integers $m_1$, $m_2$, $k$, $\ell$.  These groups $G$ are generated by three involutions $a$, $b$, and $c$ subject to the Coxeter group relations $(ab)^{m_1} = (bc)^{m_2} = (ac)^2 = 1$ and additional finiteness-ensuring relations $(abc)^k = (abcb)^{\ell}=1$.  For this string $C$-group the first vertex stabilizer subgroup $H := \langle b,c \rangle$ will be a dihedral subgroup of order $2m_2$.   We have verified with GAP \cite{GAP} that $G/\!\!/H$ has exactly two non-self-inverse double cosets when  
$$ [m_1,m_2]_{k,\ell} = [4,5]_{6,6}, [5,4]_{10,4}, [6,4]_{6,4}, [6,5]_{4,4}, [7,4]_{8,3}, [6,6]_{8,3}, [8,7]_{3,6} $$
and exactly four when 
$$\begin{array}{rcl}
[m_1,m_2]_{k,\ell} &=& [4,5]_{8,5}, [5,5]_{5,6}, [5,6]_{5,5}, [6,6]_{5,5}, [4,8]_{6,6}, [6,6]_{4,6}, \\
& & [6,6]_{5,5}, \mbox{ and } [6,6]_{6,4}. 
\end{array}$$ 
In all of these cases, the corresponding Schurian association scheme is not commutative, so Lemma \ref{Ess} applies to show its degree $2$ irreducible character has $\nu(\chi) \ge 0$.  
}
\end{example}

An explicit formula for the indicator was established for adjacency algebras of noncommutative association schemes (aka. homogeneous coherent configurations) by Higman \cite{Higman75}.  Higman's formula generalizes immediately to noncommutative RBAs with positive degree map.  It requires some of the character-theoretic constants available in the RBA setting.  If $\mathbf{B} = \{b_0=1_A, b_1, \dots, b_{r-1}\}$ is the standard basis for an RBA $(A,\mathbf{B})$ with positive degree map $\delta$, the positive real number $n = \sum_{i=0}^{r-1} \delta(b_i)$ is called the {\it order} of $(A,\mathbf{B})$.  The {\it standard feasible trace} $\tau$ of $(A,\mathbf{B})$ is defined by $\tau(\sum_{i=1}^{r-1} \alpha_i b_i) = n \alpha_0$ for all $\sum_{i=1}^{r-1} \alpha_i b_i \in A$.  This decomposes as $\tau = \sum_{\psi \in Irr(A)}  m_{\psi} \psi$, where the $m_{\psi}$ for $\psi \in Irr(A)$ are the {\it multiplicities} of $(A,\mathbf{B})$.  In \cite[Lemma 2.11 and Proposition 2.21(ii)]{Blau09}, it is shown that all of the multiplicities are positive real numbers, and that $m_{\delta}$ is always $1$.  Higman's formula, for the indicator of an irreducible character of an RBA with positive degree map, is 
$$ \nu(\chi) = \frac{m_{\chi}}{n \chi(b_0)} \sum_{i=0}^{r-1} \frac{\chi(b_i^2)}{\delta(b_i)}, \mbox{ for all } \chi \in Irr(A). $$  

Suppose $(A,\mathbf{B})$ is an RBA whose distinguished basis has exactly two nonreal elements.  By Lemma \ref{Ess} (iii), its unique degree $2$ irreducible character $\chi$ will be realized by a real representation $\Phi$.  We claim that $\chi$ is realized by a real $*$-representation.  A proof of this general fact was observed by Hanaki \cite{Hanaki-unp}.  The proof goes as follows: given a real representation $\Phi: \mathbb{C}\mathbf{B} \rightarrow M_n(\mathbb{C})$ of an RBA with positive degree map that affords $\chi$, 

(i) first show the matrix $A = \sum_{i=0}^{r-1} \frac{1}{\delta(b_i)} \Phi(b_i) \Phi(b_i^*)$ is positive definite symmetric; so

(ii) we can write $A$ as $A = BB^{\top}$ for some invertible symmetric $n \times n$ matrix $B$; and then 

(iii) $\mathcal{X} = B^{-1} \Phi B$ is a real $*$-representation affording $\chi$. 

\section{The case of two nonreal basis elements} 

Let $(A,\mathbf{B})$ be a noncommutative rank $r$ RBA with positive degree map $\delta$ whose standard basis $\mathbf{B}$ has exactly two nonreal elements.  From Lemma \ref{Ess}(iii), we know that $A$ has $r-3$ irreducible characters, one of degree $2$ that is realized over $\mathbb{R}$, and the remaining $r-2$ of them real-valued and of degree $1$.  By the remark at the end of the previous section the real representation $\mathcal{X}$ affording degree $2$ irreducible character $\chi$ can be chosen to be a $*$-representation.  If $F$ is  the field of definition for the RBA, we can see using Galois conjugacy that $\chi$ takes values in $F$.  Since $F(\chi) =F$ and we are working over a field of characteristic zero, $A_{\chi} = \mathcal{X}(F\mathbf{B})$ will be a $4$-dimensional central simple $F$-algebra. 

If $A_{\chi}$ were a division algebra, it will have dimension $4$ over its center, so it will be a generalized quaternion algebra over $F$ \cite[pg. 236]{Pierce}.  A generalized quaternion algebra over $F$ is the $4$-dimensional $F$-algebra $F1 + Fx + Fy + Fxy$ with defining relations $yx=-xy$, $x^2=\alpha$, $y^2=\beta$ for $\alpha, \beta \in F^{\times}$.  The algebra is determined by the choice of $\alpha$ and $\beta$ (in either order), so it is denoted with the symbol $\big( \dfrac{\alpha,\beta}{F} \big)$.  Up to isomorphism, this $F$-algebra is unchanged when $\alpha$ or $\beta$ are multiplied by nonzero squares in $F^{\times}$.  $\big( \dfrac{\alpha,\beta}{F} \big)$ will be isomorphic to $M_2(F)$ if either $\alpha$ or $\beta$ lie in $(F^{\times})^2$ (see \cite[\S 1.6 and \S 1.7]{Pierce}), and if the latter is not the case then the algebra is naturally isomorphic to the cyclic algebra $(F(\sqrt{\alpha})/F,\sigma,\beta)$, so it will be isomorphic to $M_2(F)$ iff $\beta$ is a norm in the extension $F(\sqrt{\alpha})/F$. 

Let $\mathbf{B} = \{ b_0,b_1,\dots,b_{r-1} \}$, with $b_i^*=b_i$ for $i = 0,1,\dots,r-3$ and $b_{r-2}^* = b_{r-1}$.  As in \cite{HMX18}, we set 
$$ \mathcal{X}(b_i) = \begin{bmatrix} r_i & s_i \\ t_i & u_i \end{bmatrix} $$ 
for some $r_i, s_i, t_i, u_i \in \mathbb{R}$, for all $b_i \in \mathbf{B}$.  Since $\mathcal{X}$ is a $*$-algebra representation, $\mathcal{X}(b_i)$ is a symmetric matrix for $i =0,1,\dots,r-3$, and $\mathcal{X}(b_{r-2})^{\top} = \mathcal{X}(b_{r-2}^*)$.

Another general fact we will need is the following: 

\begin{lemma} \label{cpcoeffs}
Suppose $F$ is the field of definition for an RBA $(A,\mathbf{B})$, and let $\chi \in Irr(A)$.  Let $\mathcal{X}$ be a representation of $A$ affording $\chi$. Then $m_{\chi} \in F(\chi)$, and for all $b_i \in B$, the coefficients of the characteristic polynomial of $\mathcal{X}(b_i)$ lie in $F(\chi)$. 
\end{lemma} 

\begin{proof} 
Since we are working over fields of characteristic zero, the extension $F(\chi)/F$ is separable, so by \cite[(74.2)]{CRII} the center of $\mathcal{X}(F\mathbf{B})$ is isomorphic to $F(\chi)$, and the coefficients of the centrally primitive idempotent $e_{\chi}$ corresponding to $\chi$ in the standard basis $\mathbf{B}$ also lie in  $F(\chi)$.  From \cite[Proposition 2.14]{Blau09}, we know we can express $e_{\chi}$ in terms of the standard basis as $$e_{\chi} = \frac{m_{\chi}}{n} \sum_{i} \frac{\chi(b_i^*)}{\delta(b_i)} b_i.$$  Since the $\delta(b_i) \in F$ for all $b_i \in \mathbf{B}$, it follows that $m_{\chi} \in F(\chi)$.  

Furthermore, we have that $\mathcal{X}(F\mathbf{B}) \simeq F(\chi)\mathbf{B}e_{\chi}$ via a map that restricts to a field isomorphism on the centers and takes $\mathcal{X}(b_i) \mapsto b_ie_{\chi}$, for all $b_i \in \mathbf{B}$.  This implies that for all $b_i \in \mathbf{B}$, $\mathcal{X}(b_i)$ maps to a matrix with entries in $F(\chi)$ under the regular representation of the algebra $\mathcal{X}(F\mathbf{B})$.  Therefore, the coefficients of the characteristic polynomial of $\mathcal{X}(b_i)$ will have coefficients in $F(\chi)$. 
\end{proof} 

The next proof makes use of two standard character identities .  Let $\Phi$ be the set of degree $1$ irreducible characters of $\mathbf{B}$.  Then $n = \tau(b_0) = 2m_{\chi} + \sum_{\phi \in \Phi} m_{\phi}$, and the fact that each $\phi \in \Phi$ is a $*$-representation of $A$ implies that for any $b_j \in \mathbf{B}$ with $b_j^* = b_j$, 
$$n \delta(b_j) = \tau(b_j^2) = m_{\chi}\chi(b_j^2) + \sum_{\phi \in \Phi} m_{\phi} \phi(b_j^2) = m_{\chi}\chi(b_j^2) + \sum_{\phi \in \Phi} m_{\phi} \phi(b_j)^2 .$$
We are now ready for our main theorem. 

\begin{thm} \label{Main}
Let $(A,\mathbf{B})$ be a noncommutative RBA with positive degree map that has two nonreal elements $b_{r-2}$ and $b_{r-2}^*$.  Let $F$ be its field of definition $F$, and let $\chi$ be its unique irreducible character of degree $2$.  Let $\mathcal{X}$ be a real $*$-algebra representation affording $\chi$.  Then there exists a $*$-invariant $b_j \in \mathbf{B}$ such that $\mathcal{X}(b_j) \not\in F \mathcal{X}(b_0)$, and for any such $b_j \in \mathbf{B}$, the symbol of $A_{\chi} \simeq \mathcal{X}(F\mathbf{B})$ can be expressed as $$\bigg( \dfrac{ -nm_{\chi} \delta(b_{r-2}), 2m_{\chi}n - m_{\chi}^2 \chi(b_j)^2 - (\sum_{\phi \in \Phi} 2 m_{\chi} m_{\phi} \phi(b_j)^2)}{F} \bigg).$$   
\end{thm} 

\begin{proof} 
Following the approach in \cite{HMX18}, define $\mathbf{D} = \{ d_0,d_1,\dots,d_{r-3}, c, d \}$ to be the basis of $A$ given by $d_i = b_i$ for $i=0,1,\dots,r-3$, $c = b_{r-2}+b_{r-2}^*$, and $d=b_{r-2}-b_{r-2}^*$.  Then $\{d_0,d_1,\dots,d_{r-3},c\}$ is a basis for the subspace of $A$ consisting of $*$-invariant elements, and $d^* = - d$.  It then follows that 
$$\begin{array}{rl} 
\tau(d d^*) =& - \tau(d^2) \\
=& -\tau(b_{r-2}^2 + b_{r-2}^{*2} - b_{r-2}b_{r-2}^* - b_{r-2}^*b_{r-2}) \\
=& -0 - 0 + n\delta(b_{r-2}) + n\delta(b_{r-2}) \\
=& 2n\delta(b_{r-2}). 
\end{array}$$
Recall from Lemma \ref{Ess}(iii) that every irreducible character of $A$ is real-valued.  This means every irreducible character of degree $1$ will vanish on $d^2$.  Since $\mathcal{X}(d) = \begin{bmatrix} 0 & s_{r-2}-t_{r-2} \\ t_{r-2}-s_{r-2} & 0 \end{bmatrix}$, we have 
$$\begin{array}{rl} 
\tau(d d^*) =& m_{\chi} \chi(d d^*) \\
=& m_{\chi}tr(\mathcal{X}(d) \mathcal{X}(d)^{\top}) \\
=&  2 m_{\chi} (s_{r-2} - t_{r-2})^2. 
\end{array}$$
So it follows that $m_{\chi} (s_{r-2}-t_{r-2})^2 = n \delta_{r-2}.$  Therefore, $$\mathcal{X}(d)^2 = \bigg(\dfrac{-n\delta(b_{r-2})}{m_{\chi}} \bigg) \mathcal{X}(b_0). $$ 
In particular, if we take $x = m_{\chi}\mathcal{X}(d)$ to be one of the generators of $A_{\chi}$ as a quaternion algebra, then $x^2 = -nm_{\chi}\delta(b_{r-2}) I$. 

We now need to find a suitable choice for the other symbol algebra generator $y$.  Since $A_{\chi}$ has dimension $4$, there are at least two of our other nontrivial $*$-invariant elements $d_{\ell} \in \{d_1,\dots,d_{r-3},c\}$ for which $\mathcal{X}(d_{\ell})$ has distinct eigenvalues.  So we can choose this $d_{\ell}$ to be one of the real elements $b_j$ of $\mathbf{B}$.  

After conjugating by a suitable orthogonal matrix (which will not affect the form of $\mathcal{X}(d)$ or the result for $x^2$ above), we can assume $\mathcal{X}(b_j) = \begin{bmatrix} r_j & 0 \\ 0 & u_j \end{bmatrix} \mbox{ with } r_j \ne u_j.$  Since $x$ has the form $\begin{bmatrix} 0 & \alpha \\ -\alpha & 0 \end{bmatrix}$ for some $\alpha \in F^{\times}$, it is easy to see that in order for a symmetric matrix $y=\begin{bmatrix} r & s \\ s & u \end{bmatrix}$ to satisfy $yx = -xy$, we only require $u = -r$.  Setting $y = 2\mathcal{X}(b_j) - (r_j + u_j)\mathcal{X}(b_0)$ accomplishes this, and for this $y$ we have that $y^2 = [(r_j - u_j)^2] I$ is a positive multiple of the identity.  Furthermore by Lemma \ref{cpcoeffs} the trace  $r_j+u_j=\chi(b_j)$ and determinant $r_ju_j$ both lie in $F(\chi)=F$, and hence $y^2 = (r_j-u_j{\ell})^2I = [(r_j+u_j)^2-4r_ju_j] I \in FI$.   Therefore, $A_{\chi}$ has the symbol $\bigg( \dfrac{-nm_{\chi}\delta(b_{r-2}),(r_j-u_j)^2}{F} \bigg)$.  

It remains to give a character-theoretic expression for $(r_j-u_j)^2$.   If $a = \frac{\chi(b_j)}{2}$, then assuming without loss of generality that $r_j > u_j$, we have $(r_j - u_j) = 2\varepsilon$ for some $\varepsilon > 0$.  Furthermore, $r_j^2 + u_j^2 = 2a^2 + 2 \varepsilon^2$, so 
$$ (r_j - u_j)^2 = 4 \varepsilon^2 = 2(r_j^2 + u_j^2) - 4 a^2 = 2 \chi(b_j^2) - \chi(b_j)^2. $$ 
Since $b_j$ is $*$-invariant, $m_{\chi} \chi(b_j^2) = n \delta(b_j) - \sum_{\phi \in \Phi} m_{\phi} \phi(b_j)^2$.  Therefore, 
$$(r_j-u_j)^2 = [\frac{2}{m_{\chi}}(n \delta(b_j) - \sum_{\phi \in \Phi} m_{\phi} \phi(b_j)^2) - \chi(b_j)^2].$$  Multiplying this by $m_{\chi}^2$ produces the desired second parameter in the symbol given for $A_{\chi}$.  
\end{proof} 

\begin{cor} 
Suppose $\mathbf{B}$ is the standard basis of a noncommutative rank $5$ RBA with positive degree map, and let $F$ be its field of definition. Let $\mathcal{X}$ be a $*$-representation affording the degree $2$ irreducible character $\chi$ of $\mathbf{B}$.  Then there will be at least one  $*$-invariant element $b_j \in \mathbf{B}$ for which $\mathcal{X}(b_j)$ has distinct eigenvalues, and for any such $b_j$ the symbol of $A_{\chi}$ is equivalent to $\big( \dfrac{ -\frac12 n(n-1)\delta(b_3), (n-1)\delta(b_j) - \delta(b_j)^2 }{F} \big)$. 
\end{cor}

\begin{proof}
The main results of \cite{HMX17} show these RBAs are determined by their degree map.  The values of the degree $2$ irreducible character are $\chi(b_i) = \frac{-2 \delta(b_i)}{n-1}$ for all $b_i \in \mathbf{B}$, and $m_{\chi} = \frac{n-1}{2}$.  Using the fact that $\Phi = \{\delta\}$ in Theorem \ref{Main} and substituting these values into the formula for the symbol of $A_{\chi}$ gives the result. 
\end{proof} 

\begin{example} {\rm Example 11 from \cite{HMX17} gives a noncommutative rational rank $5$ table algebra with order $25$ whose nontrivial basis elements have degree $6$, that has a $*$-invariant element $b_1$ with 
$$\mathcal{X}(b_1) = \frac14 \begin{bmatrix} -1+5\sqrt{3} & 0 \\ 0 & -1-5 \sqrt{3} \end{bmatrix}.$$  
So in this case we will have $m_{\chi} = 12$, and the symbol for $A_{\chi}$ in the above corollary has parameters $\alpha = -(25)(12)(6) = -2 (30)^2$ and $\beta = (12)^2 (\frac{(25)(3)}{4}) = 3(15)^2$. 
So $A_{\chi}$ has symbol $\big( \dfrac{-2, 3}{\mathbb{Q}} \big)$.  Using the Legendre symbol as in \cite[\S 1.7]{Pierce} to compute the local Schur indices of this generalized quaternion algebra over $\mathbb{Q}$, we see that the local index will be $2$ at the primes $p=2$ and $3$.  (This can also be accomplished with the {\tt LocalIndicesOfRationalQuaternionAlgebra} command in the GAP package {\tt wedderga} \cite{wedderga}.)  So in this case $A_{\chi}$ is a $4$-dimensional $\mathbb{Q}$-division algebra.  }
\end{example}

\begin{example} {\rm Table 1 of \cite{HMX18} gives character-theoretic information for many integral rank $6$ RBAs with positive degree map having order up to 150.  For these the field of definition is $\mathbb{Q}$.  In checking this table we find examples for which $A_{\chi} = \big( \dfrac{ \alpha,\beta}{\mathbb{Q}} \big)$ is a noncommutative division algebra.

The first of these is the eleventh listed in the table, with order $n=33$.  From the information provided we can use the identities of \cite{HMX18} to produce its character table:  
$$\begin{array}{r|cccccc|l}
         & b_0 & b_1 & b_2 & b_3 & b_4 & b_4^* & m_{\psi} \\ \hline
\delta &    1 &   10 &   10 &  10  &    1 &     1     &    1 \\
\phi   &    1  &   -1 &   -1  &  -1  &    1 &     1     &  10 \\
\chi    &   2  &    0  &   0   &   0  &   -1 &    -1    &   11 \\
\end{array} $$
Our formulas in Theorem \ref{Main} tell us $\alpha = -n m_{\chi} \delta(b_4) = -3(11)^2$ and 
$\beta = (11)^2(\frac{2}{11}((33)(10)-(10^2 +10(-1)^2 + 11(0)^2) - (0)^2) = (11)^2 (40)$, so $A_{\chi} \simeq \big( \dfrac{ -3 , 10}{\mathbb{Q}} \big)$.  This rational quaternion algebra has local index $2$ at the primes $p=2$ and $5$, so it is a $4$-dimensional division algebra. 
}
\end{example}

It is interesting to note that the above examples correspond to a table algebras that do not arise from association schemes.  Although we did not find a noncommutative rank $6$ association scheme in Table 1 of \cite{HMX18} that produces a nontrivial rational Schur index, the one primitive table algebra of order $81$ whose feasibility as an association scheme is open does have this property:   

\begin{example} {\rm Let $\mathbf{B}$ be the second example of order $81$ in Table 1 of \cite{HMX18}, which has character table 
 $$\begin{array}{r|cccccc|l}
         & b_0 & b_1 & b_2 & b_3 & b_4 & b_4^* & m_{\psi} \\ \hline
\delta &    1  &  10 &   10 &  20  &   20 &    20    &    1 \\
\phi   &    1  &    1 &    1  &  -7  &     2 &     2     &  20 \\
\chi   &    2  &   -1 &   -1  &   4  &    -2 &    -2     &  30 \\
\end{array} $$
Therefore, $\alpha = - (81)(30)(20) = -6 (90)^2$ and $\beta = (30)^2 (\frac{2}{30}((81)(10)-(10^2 - 20(1)^2) - (-1)^2) = 5(90)^2$, so $A_{\chi} \simeq \big( \dfrac{-6,5}{\mathbb{Q}} )$, which has local Schur index $2$ at the primes $p=2$ and $3$. 
} \end{example} 

\section{Noncommutative RBAs of rank $7$} 

Noncommutative RBAs of rank $5$ and $6$ with positive degree map were discussed in detail in \cite{HMX17} and \cite{HMX18}.  In both cases Frobenius-Schur indicator theory implies that $\mathbf{B}$ has only $2$ nonreal elements, so our main theorem applies to them.  However, the Frobenius-Schur indicator argument is inconclusive in the rank $7$ case.  Suppose $F$ is the field of definition for a noncommutative rank $7$ RBA with positive degree map.  If the irreducible characters are $\delta$, $\phi$, $\psi,$ and $\chi$, where $\delta$ is the postive degree map, $\phi(b_0) = \psi(b_0) = 1$, and $\chi(b_0)=2$, then $\delta$ and $\chi$ take values in $F$.  Therefore, the indicator values are $\nu(\delta)=1$, $\nu(\phi) = \nu(\psi) = 0$ or $1$, and $\nu(\chi)=-1$ or $1$.   The possible numbers of $*$-invariant elements of $\mathbf{B}$ are 
$$\begin{array}{l} 
5, \mbox{ when } \nu(\phi) = \nu(\psi)=1 \mbox{ and } \nu(\chi)=1 \\
3, \mbox{ when } \nu(\phi) = \nu(\psi)=0 \mbox{ and } \nu(\chi)=1, \mbox{ and } \\
1, \mbox{ when } \nu(\phi) = \nu(\psi) = 1 \mbox{ and } \nu(\chi) = -1. 
\end{array}$$

In searching the database of small association schemes, we find several examples noncommutative association schemes of rank $7$ whose standard basis has $2$ nonreal elements.  It is possible to construct noncommutative rank $7$ table algebras with $4$ nonreal basis elements as wreath products of rank $3$ table algebras with $2$ nonreal elements with the noncommutative rank $5$ table algebras discussed in \cite{HMX17}, but we have neither been able to find nor rule out the existence of noncommutative rank $7$ association schemes with $4$ nonreal elements.  

As for noncommutative rank $7$ RBAs whose basis has $6$ nonreal elements, the above shows their character table will have the form 
$$\begin{array}{c|ccccccc|c}
\quad & b_0 & b_1 & b_1^* & b_2 & b_2^* & b_3 & b_3^* & m \\ \hline 
\delta & 1 & \delta_1 & \delta_1 & \delta_2 & \delta_2 & \delta_3 & \delta_3 & 1 \\
\phi   & 1 & \phi_1 & \phi_1 & \phi_2 & \phi_2 & \phi_3 & \phi_3 & m_{\phi} \\
\psi   & 1 & \psi_1 & \psi_1 & \psi_2 & \psi_2 & \psi_3 & \psi_3 & m_{\psi} \\
\chi & 2 & \chi_1 & \chi_1 & \chi_2 & \chi_2 & \chi_3 & \chi_3 & m_{\chi} 
\end{array}$$ 

\begin{thm} 
Noncommutative rank $7$ RBAs with positive degree map whose degree $2$ irreducible character has $\nu(\chi) = -1$ cannot have integral structure constants.   In particular, there are no association schemes of this type. 
\end{thm} 

\begin{proof}  
Consider the character table above.  The first orthogonality relation implies that the sums of the values of the irreducible characters other than the degree map will be $0$.  For $\phi$ (and $\psi$) this implies $1 + 2 (\phi_1 + \phi_2 +\phi_3) = 0$, and so $\phi_1 + \phi_2 + \phi_3 = -\frac12$.  If the RBA is integral, the values $\phi_1, \phi_2,$ and $\phi_3$ must lie in the ring $\mathcal{O}$ of algebraic integers of the field of character values $\mathbb{Q}(\phi)$.  But then the sum of these numbers would have to be an algebraic integer, so it could not be $-\frac12$. This contradiction implies that any RBA of this type will not have integral structure constants.  
\end{proof}

Non-integral examples satisfying the assumptions of the previous theorem do exist.  It is relatively straightforward to produce an admissible character table (with rational character values and multiplicities), and use it to construct a representation affording $\chi$.  For example, we have constructed a noncommutative rank $7$ RBA with postive degree map that has this (admissible) character table:  
$$\begin{array}{c|ccccccc|c}
\quad & b_0 & b_1 & b_1^* & b_2 & b_2^* & b_3 & b_3^* & m \\ \hline 
\delta & 1 & 2 & 2 & 2 & 2 & 2 & 2 & 1 \\
\phi   & 1 & -\frac52 & -\frac52 & 0 & 0 & 2 & 2 & 52/45  \\
\psi   & 1 & 2 & 2 & -\frac92 & -\frac92 & 2 & 2 & 4/9 \\
\chi & 2 & 0 & 0 & 0 & 0 & -1 & -1 & 26/5
\end{array}$$ 
Its irreducible character of degree $2$ is realized by a $*$-representation $\mathcal{X}: \mathbb{R}\mathbf{B} \rightarrow \mathbb{H}$ that has 
$$\mathcal{X}(b_0) = 1, \mathcal{X}(b_1) = \frac{\sqrt{5}}{2}i, \mathcal{X}(b_2) = \frac{\sqrt{5}}{2}j, \mbox{ and } \mathcal{X}(b_3) = \frac12 + \frac{\sqrt{5}}{2}k. $$  (Here the real quaternion algebra $\mathbb{H}$ is considered in its usual basis $\{1,i,j,k\}$ with the involution $(t + xi + yj + zk)^* = t - xi - yj - zk$.  Its has local index $2$ at the primes $2$ and $\infty$.) 

\bigskip
The author would like to thank the referee for pointing out an error in the submitted version of the article that led to a significant revision of the statement of the main theorem and to the division algebra examples.  The author would also like to express his gratitude to Harvey Blau directing the author to reference \cite{LM2000}, and for some comments that improved the presentation of some arguments here.

\end{document}